\documentclass{article}
\usepackage{amsmath,amssymb,amsfonts,amsthm,mathrsfs}

\usepackage{xcolor}
\usepackage{graphicx} 
\usepackage{mathtools} 
\usepackage[margin=20pt]{caption} 

\usepackage[style=alphabetic,sorting=nyt,maxbibnames=7,maxcitenames=5, backend=biber, style=alphabetic]{biblatex}
\usepackage{thmtools, thm-restate} 
\usepackage{hyperref} 
\usepackage{cleveref} 

\newtheorem{thm}{Theorem}[section]

\newtheorem{prop}[thm]{Proposition}
\newtheorem{lem}[thm]{Lemma}
\newtheorem{cor}[thm]{Corollary}

\theoremstyle{definition}
\newtheorem{dfn}[thm]{Definition}
\newtheorem{rem}[thm]{Remark}
\AtEndEnvironment{rem}{\null\hfill\ensuremath{\triangle}}

\newtheorem{mainthm}{Theorem}

\newtheorem{maincor}[mainthm]{Corollary}

\newcommand{\NN}{\mathbb N}

\newcommand{\ZZ}{\mathbb Z}
\newcommand{\Sph}{\mathbb S}
\newcommand{\Sp}[1]{\mathbb{S}^{#1}}
\newcommand{\rev}{^\mathrm{rev}} 

\newcommand{\h}{\operatorname{H}} 

\newcommand{\lk}{\operatorname{lk}} 
\newcommand{\fr}{\operatorname{fr}} 

\newcommand{\g}{\operatorname{g}}

\newcommand\extrafootertext[1]{%
    \bgroup
    \renewcommand\thefootnote{\fnsymbol{footnote}}%
    \renewcommand\thempfootnote{\fnsymbol{mpfootnote}}%
    \footnotetext[0]{#1}%
    \egroup
}

\title{Genus-zero links\\with prescribed knots as components}
\author{Raphael Appenzeller \and José Pedro Quintanilha}
\date{}

\begin{document}

\maketitle

\begin{abstract}
    We prove that any finite collection of at least three isotopy classes of knots in a \(3\)-manifold~\(M\) is realizable as the components of a genus-zero link in~\(M\), provided that an obvious requirement on their conjugacy classes in \(\pi_1(M)\) is met. 
    This condition is vacuously satisfied for \(M=\Sph^3\), and in this case we also control the pairwise linking numbers of the components.
    Replacing the \(3\)-genus with the \(4\)-genus, we obtain an analogous result where only two knot isotopy classes are prescribed.
\end{abstract}

\extrafootertext{2020 \textit{Mathematics Subject Classification}: 57K10, 57M10.}
\extrafootertext{\textit{Keywords}: knots, Seifert surface, link genus, three-manifolds, low dimensional topology.}
\extrafootertext{\textit{Date}: \today.} 



\section{Introduction}

A knot in~\(\Sph^3\) is a genus-zero link if and only if it is an unknot.
For two knots~\(K\),~\(K'\), the existence of a genus-zero link with components isotopic to \(K\)~and~\(K'\) is equivalent to \(K'\)~being isotopic to the orientation-reversed~\(K\rev\). Once three knots are given, however, the situation becomes much less restricted, as \Cref{fig:beautiful_knots} illustrates.

\begin{figure}[h!]
    \centering
    \def\svgwidth{0.46 \linewidth}
    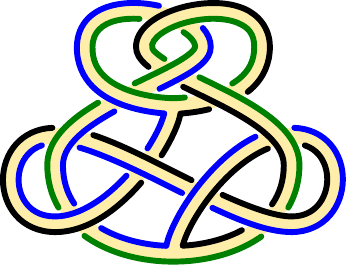
    \caption{A genus-zero surface in~\(\Sph^3\) whose boundary components are a left-handed trefoil (blue), a right-handed trefoil (black), and a figure-eight knot (green).}
    \label{fig:beautiful_knots}
\end{figure}

In this note, we show that any multiset of at least three (oriented) knots may be realized as a genus-zero link. In fact, our main theorem allows for the given knots to live in any \(3\)-manifold~\(M\), as long as a necessary homotopy-theoretical condition is met:

\begin{mainthm}[Genus-zero links of prescribed knots]\label{main_intro}
    Fix \(n\ge 3\) knots \(K_1, \dots, K_n\) in a connected \(3\)-manifold~\(M\), and for each~\(i\le n\) let \(C_i\subseteq\pi_1(M)\) be the conjugacy class represented by~\(K_i\). The following are equivalent:
    \begin{enumerate}
        \item The conjugacy classes \(C_1, \dots, C_n\)~form a canceling multiset.
        \item  There is an \(n\)-component genus-zero link whose components are isotopic to \(K_1 , \ldots , K_n\).
    \end{enumerate}    
\end{mainthm}

A multiset of conjugacy classes is ``canceling'' if these classes can by represented by elements whose product is trivial (\Cref{def:cancel}). In \Cref{sec:canceling_multiset} we give a topological interpretation of this condition, which explains why its failure obstructs the existence of a genus-zero link with the prescribed knots as components (\Cref{topological-cancellation}).

\Cref{main_intro} is restated and proved in \Cref{sec:main_proof} as \Cref{main}. The argument for the implication \((1\Rightarrow 2)\) proceeds in two steps. First, we find an embedded genus-zero surface \(S_0\subset M\) with \(n\)~boundary components that are merely homotopic to \(K_1, \dots, K_n\) (\Cref{starting_surface}). Second, we explain how to modify these boundary components so they become knots in the correct isotopy classes (\Cref{change_the_boundary_knots}).

One should generally expect the components of a link~\(L\) as produced by \Cref{main_intro} to be intricately linked. 
However, in \Cref{sec:linking} we show that for \(M=\Sph^3\) it is possible to fully control their pairwise linking numbers (\Cref{adjust_linking}), proving (as \Cref{full_linking_result}):

\begin{maincor}[Genus-zero links in \(\Sph^3\) of prescribed knots and linking numbers]\label{thm:linking_intro}
    Let \(K_1, \dots, K_n\) be \(n \geq 3\) knots 
    in~\(\Sph^3\), and let \((\ell_{ij})_{i,j}\) be a symmetric \(n\times n\) matrix of integers. The following are equivalent:
    \begin{enumerate}
        \item For each \(i\in\{1, \dots, n\}\), we have
        \(\sum_{j=1}^n \ell_{ij}= 0.\)
        \item There is a genus-zero link in \(\Sph^3\) with components \(K_1', \dots, K_n'\), such that for all \(i,j\in \{1, \dots, n\}\),
        \begin{enumerate}
            \item the knot \(K_i'\) is isotopic to~\(K_i\),
            \item if \(i \neq j \) then the linking number of $K_i'$ and $K_j'$ is $\ell_{ij}$, and
            \item the framing of $K_i'$ induced by some / every Seifert surface is $\ell_{ii}$. 
        \end{enumerate}
    \end{enumerate}
\end{maincor}

In \Cref{sec:n=2} we discuss the case of \(n=2\) prescribed knots \(K_1\), \(K_2\). While there may not be a genus-zero link with components isotopic to \(K_1\), \(K_2\), we show that there is a link of vanishing \(4\)-genus (\Cref{cor:zero_4genus}):

\newpage 

\begin{maincor}[Links with vanishing \(4\)-genus of prescribed knots]\label{maincor:zero_4genus}
    For all knots~\(K_1\), \(K_2\subset M\), the following are equivalent:
    \begin{enumerate}
        \item \(K_1\)~is homotopic to~\(K_2\rev\),
        \item There is a two-component link~\(L\subset M\) with components isotopic to~\(K_1\),~\(K_2\), such that \(\g_4(L)=0\). 
    \end{enumerate} 
\end{maincor}

\subsection*{Acknowledgements}
We are grateful to Sebastian Baader for pointing out the interpretation of band surgery in terms of embedded surfaces, when we were thinking about a magic trick involving knots, see \Cref{rem:band_surgery}. We are also grateful for inspiring discussions with Stefan Friedl and Paula Truöl. The authors thank the Inkscape community for developing and maintaining the free and open-source vector graphics software used to prepare the figures in this work.

\subsection*{Statement of originality} 
Near the end of the drafting of this manuscript, the authors noticed that the statement of \Cref{main_intro} in the case \(M=\Sp3\) had already appeared in a retracted preprint of Murakami \cite[Proposition~6.1]{Mur06}. The argument therein is different from the one we present, resting on the main result of an article of Kinoshita \cite{Kin87} about realizing prescribed knots as the cycles in \(\theta\)-curves, which are a certain type of spatial graphs.

\section{Canceling conjugacy classes}\label{sec:canceling_multiset}

We begin with a discussion on the \(\pi_1\)-condition featured in~\Cref{main_intro}, which obstructs a collection of isotopy classes of knots in a \(3\)-manifold assembling to a genus-zero link.
Since knots are typically not pointed objects, they do not represent elements of the fundamental group, but only conjugacy classes. While one cannot generally speak of multiplication of conjugacy classes as a well-defined operation, we will rescue the idea of them multiplying to the trivial element.
Before the proper definition, we show this is an order-independent notion.

\begin{lem}[Cancellation is order independent]\label{order_does_not_matter}
    Let \(C_1, \dots, C_n\) be conjugacy classes in a group~\(G\), and let \(\sigma\in \operatorname{Sym}(n)\) be a permutation. If there are \(g_i\in C_i\) for every \(i\in \{1,\dots, n\}\) such that \(g_1\cdots g_n = 1\), then there also exist \(g_i'\in C_i\) with \(g_{\sigma(1)}'\cdots g_{\sigma(n)}' = 1\).
\end{lem}
\begin{proof}
    For every \(k\in \{1,\ldots, n-1\}\), the equality \(g_1 \cdots g_n = 1\) is equivalent to
    \[g_1 \dots (g_kg_{k+1}g_k^{-1}) g_k g_{k+2}\cdots g_n=1.\]
    Thus, for the transposition \(\sigma = (k\quad k+1)\), the statement holds with \(g_{k+1}'\coloneq g_{k}g_{k+1}g_k^{-1}\) and \(g_i'\coloneq g_i\) for \(i\neq k+1\).
    As the adjacent transpositions generate \(\operatorname{Sym}(n)\), the conclusion extends to all permutations.
\end{proof}

\Cref{order_does_not_matter} ensures that the following is a well-defined notion for an (unordered) multiset of conjugacy classes.

\begin{dfn}\label{def:cancel}
    A multiset of conjugacy classes \(C_1, \dots, C_n\) in a group~\(G\) is \textbf{canceling} if there are \(g_1 \in C_1\), \ldots , \(g_n\in C_n\) such that \(g_1\cdots g_n = 1\).
\end{dfn}

    

We now give a topological interpretation of the property of a multiset of conjugacy classes being canceling.
Recall that a free homotopy class of a loop~\(\gamma\) in a path-connected topological space~\(X\) represents a conjugacy class in \(\pi_1(X,p)\) with respect to any basepoint \(p\in X\), which is given by choosing a path~\(\alpha\) in~\(X\) from~\(p\) to the starting point of~\(\gamma\), and taking the pointed homotopy class of the concatenation \(\alpha * \gamma *\alpha^{-1}\). For different choices of~\(\alpha\) the corresponding elements of $\pi_1(X,p)$ differ by an inner automorphism of \(\pi_1(X,p)\), and every element in the conjugacy class may be produced in this way. Moreover, given another basepoint \(p'\in X\), we obtain an isomorphism \(\pi_1(X,p') \cong \pi_1(X,p)\) by applying the same construction with a path~\(\alpha\) from \(p'\) to~\(p\). This allows us to speak of the (basepoint-free) homotopy class of a loop~\(\gamma\) as ``a conjugacy class in~\(\pi_1(X)\)'' without making basepoints explicit.

\begin{prop}[Cancellation, topologically]\label{topological-cancellation}
    Let \(\Sigma_n\) be an oriented compact connected genus-zero surface with \(n\)~boundary components \(\gamma_1, \dots, \gamma_n\). Let  \(X\) be a nonempty path-connected topological space, and \(C_1, \dots,C_n\) conjugacy classes in \(\pi_1(X)\). The following are equivalent:
    \begin{enumerate}
        \item The classes \(C_1, \dots, C_n\) form a canceling multiset.
    
        \item There is a continuous map \(f\colon \Sigma_n \to X\) such that for each \(i\in \{1, \dots, n\}\), the oriented loop~\(f(\gamma_i)\) represents~\(C_i\).
    \end{enumerate}
\end{prop}
\begin{proof}
    Before directly addressing either of the implications, we introduce a CW structure on~\(\Sigma_n\), illustrated in \Cref{fig:CW_structure}. There is one \(0\)-cell~\(p\) in the interior of~\(\Sigma_n\), and one in each boundary component~\(\gamma_i\). Each~\(\gamma_i\) is a \(1\)-cell, and there are also \(1\)-cells~\(\alpha_i\) connecting~\(p\) to~\(\gamma_i\). These arcs~\(\alpha_i\) emanate from~\(p\) in a cyclic order opposite to the orientation of~\(\Sigma_n\). Lastly, there is a \(2\)-cell attached along the loop \((\alpha_1*\gamma_1*\alpha_1^{-1})*\cdots*(\alpha_n *\gamma_n*\alpha_n^{-1})\).
    
    \begin{figure}[h!]
        \centering
        \def\svgwidth{0.38\linewidth}
        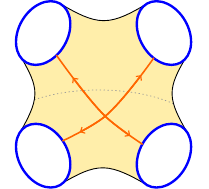
        \caption{A CW structure on the surface~\(\Sigma_n\), exemplified with \(n=4\). }
        \label{fig:CW_structure}
    \end{figure}

    \((2 \Rightarrow 1)\) Since the property of being a canceling multiset is preserved under group homomorphisms, it suffices to show that the~\(\gamma_i\) represent a canceling multiset of conjugacy classes in \(\pi_1(\Sigma_n)\). We represent these classes as loops based at~\(p\) using the arcs \(\alpha_1, \dots, \alpha_n\); explicitly, we consider the pointed homotopy classes of the loops \(\alpha_i * \gamma_i *\alpha_i^{-1}\) in~\(\pi_1(X,p)\). These group elements satisfy
    \[(\alpha_1 * \gamma_1 *\alpha_1^{-1}) * \cdots* (\alpha_n * \gamma_n *\alpha_n^{-1}) =1,\]
    since this path gives the attaching map for the \(2\)-cell in~\(\Sigma_n\). 

    \((1 \Rightarrow 2)\) Choose \(x\in X\), and for every \(i\in \{1,\dots, n\}\) choose a loop \(\delta_i\) in~\(X\) representing \(C_i\). Next, choose for each~\(i\) a path \(\beta_i\) in~\(X\) from \(x\) to the starting point of~\(\delta_i\), such that in \(\pi_1(X,x)\) we have
    \[(\beta_1*\delta_1*\beta_1^{-1}) * \cdots*(\beta_n*\delta_n*\beta_n^{-1}) = 1;\]
    the existence of the~\(\beta_i\) is ensured by Condition~1.

    We define the map~\(f\) on the \(1\)-skeleton~\(\Sigma_n^{(1)}\) by sending~\(p\) to~\(x\), each boundary component~\(\gamma_i\) to~\(\delta_i\), and each arc~\(\alpha_i\) to~\(\beta_i\) in the obvious way. The loop along which the \(2\)-cell of~\(\Sigma_n\) is attached is therefore mapped to the \(\pi_1\)-trivial loop in the above formula. This means \(f\)~may be extended from~\(\Sigma^{(1)}_n\) to all of~\(\Sigma_n\).
\end{proof}

\section{Proof of \Cref{main_intro}}\label{sec:main_proof} 

Throughout this article, manifolds and embeddings are understood to be smooth.  
For this section, we fix a connected \(3\)-manifold~\(M\). A \textbf{link} in~\(M\) is the oriented image of an embedding of a finite disjoint union of copies of $\Sp1$ in the interior~\(\mathring M \coloneq  M\setminus \partial M\). Link isotopies in~\(M\) are assumed to be smooth and to have image in~\(\mathring M\).
A \textbf{knot}~\(K\) is a one-component link; we denote by~\(K\rev\) the knot~\(K\) with its orientation reversed. A \textbf{Seifert surface} for a link~$L$ is an oriented compact connected surface~\(S\) embedded in~\(M\) with $\partial S =L$ (as oriented manifolds).
If \(L\)~has a Seifert surface~\(S\), the \textbf{genus} of~\(L\) is the minimal genus of such an~\(S\).

\subsection{Prescribing boundary homotopy classes}\label{sec:hom_prescribe}

The first step in proving \Cref{main_intro} is constructing an embedded surface in~\(M\) whose boundary components lie in the correct homotopy classes. In the case \(M= \Sph^3\), any embedded surface satisfies this condition.

\begin{prop}[Prescribing boundary homotopy classes]\label{starting_surface}
    For every \(n \in \NN\) and every canceling multiset of conjugacy classes \(C_1, \dots, C_n\)  in~\(\pi_1(M)\), there exists an oriented embedded genus-zero surface \(S \subset M\) with \(n\) boundary components, each representing one of the given conjugacy classes.
\end{prop}
\begin{proof}
    If \(n=0\), we can take \(S\) to be any embedded \(2\)-sphere in \( M\), so assume from now on that \(n\ge 1\). The construction of~\(S\) in the case \(n=4\) is illustrated in \Cref{fig:starting_surface}.

    \begin{figure}[h!]
    \centering
    \def\svgwidth{0.74\linewidth}
    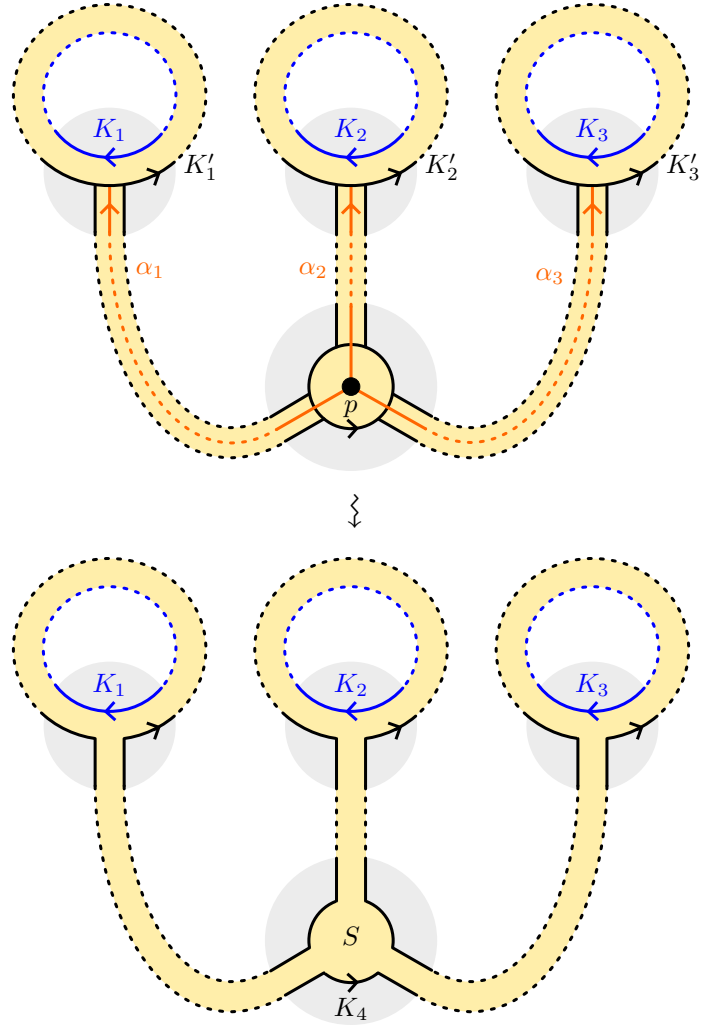
    \caption{The construction of~\(S\) for \(n=4\). The dotted regions might be linked, knotted, and twisted within the ambient manifold~\(M\).}
    \label{fig:starting_surface}
\end{figure}

    Choose pairwise disjoint knots \(K_1', \dots, K'_{n-1} \subset M\) representing the classes \(C_1^{-1}, \dots, C_{n-1}^{-1}\), respectively.
    We also fix a basepoint \(p\in M\), so the~\(C_i\) are regarded as subsets of~\(\pi_1(M, p)\). By assumption, there are \(g_i\in C_i\) with \(g_1\cdots g_n = 1\). 
    Moreover, for each \(i\in \{1,\dots,n-1\}\) there is an embedded arc \(\alpha_i \subset M\) from~\(p\) to a point of~\(K'_i\), such that the loop at~\(p\) given by the concatenation \(\alpha_i * K'_i* \alpha_i^{-1}\) represents~\(g_i^{-1}\). Up to a small isotopy relative endpoints, we may assume that each arc~\(\alpha_i\) has interior disjoint from all the other arcs, and from all the knots.
    
    Choose a small oriented embedded disk \(D\subset M\) containing~\(p\) in its interior. With an additional isotopy of each~\(\alpha_i\), we may arrange that \(\alpha_i \cap D\)~is an initial segment of~\(\alpha_i\) connecting~\(p\) to~\(\partial D\) and, in addition, the (oriented) circle \(\partial D\) meets the arcs \(\alpha_i\) in the cyclic order \(\alpha_{n-1}, \alpha_{n-2}, \dots, \alpha_1, \alpha_{n-1}\).

    Next, we choose for each \(i\in\{1, \dots, n-1\}\) a thin oriented embedded annulus \(A_i \subset M\) having~\(K_i'\) as a boundary component, and otherwise disjoint from the rest of the construction. The second boundary component of~\(A_i\) is a knot~\(K_i\) isotopic to~\((K_i')\rev\), and hence represents the conjugacy class~\(C_i\).

    Finally, we use the arcs~\(\alpha_i\) to produce thin bands~\(B_i\subset M\) joining~\(D\) to the~\(A_i\), in a manner consistent with orientations, so that the union
    \[S \coloneq D \cup\biggl( \bigcup_{i=1}^{n-1} A_i \biggr) \cup \biggl(\bigcup_{i=1}^{n-1}B_i\biggr)\]
    is an oriented embedded surface of genus \(0\) in~\(M\). Its boundary consists of the knots \(K_1, \dots, K_{n-1}\), and one last component~\(K_n\), homotopic to the loop
    \[(\alpha_{n-1}*K_{n-1}'*\alpha_{n-1}^{-1}) * \dots * (\alpha_1*K_1'*\alpha_1^{-1}).\]
    Hence, \(K_n\)~represents the conjugacy class of \(g_{n-1}^{-1}\cdots g_1^{-1} = g_n\), which is to say~\(C_n\).
\end{proof}

\begin{rem}(Pointing the knots)
    In the above proof, the role of the arcs \(\alpha_1, \dots, \alpha_{n-1}\) is not only to make the knots~\(K_1,\dots, K_{n-1}\) into pointed objects, so they represent actual elements of~\(\pi_1(M,p)\), but also to choose elements in the conjugacy classes~\(C_1, \dots, C_{n-1}\) whose product lies in~\(C_n^{-1}\). It is easy to construct examples of \(3\)-manifolds~\(M\) where (say) three knots \(K_1, K_2, K_3\) share a base-point, lie in conjugacy classes \(C_1, C_2, C_3\) that form a canceling multiset, but whose corresponding elements~\(x_1, x_2, x_3\in \pi_1(M)\) do not satisfy \(x_1x_2x_3=1\). One need only choose two non-commuting elements \(x,y\in \pi_1(M,p)\) and pointed knots~\(K_1, K_2, K_3\) representing the elements \(x,x^{-1},[y,x]\) respectively. 
\end{rem}

\subsection{Adjusting the boundary knots}\label{sec:adjusting}

We say that two knots \(K\), \(K'\) in~\(M\) differ by a \textbf{crossing change} if there exists an embedded closed \(3\)-ball \(B\subset M\) with \(\partial B\) intersecting~\(K\) transversely, such that \(B\cap K\) consists of two arcs in~\(B\) forming a trivial tangle, and \(K'\) is obtained from~\(K\) by modifying this tangle as illustrated in \Cref{fig:3d_crossing_change}. Up to isotopy, \(B\)~may be taken to be an arbitrarily small closed neighborhood of an arc~\(\alpha\) connecting the two strands of~\(K\) (and otherwise disjoint from~\(K\)). In this case, we say that $\alpha$ \textbf{guides} the crossing change.

\begin{figure}[h!]
    \centering
    \def\svgwidth{0.7\linewidth}
    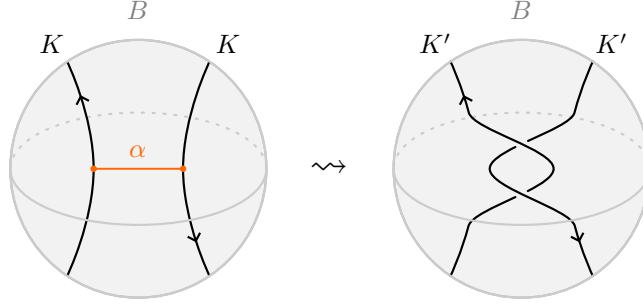
    \caption{A crossing change on a knot~\(K\subset M\), supported on the \(3\)-ball~\(B\), which may be taken to be an arbitrarily small closed neighborhood of the arc~\(\alpha\).}
    \label{fig:3d_crossing_change}
\end{figure}

If two knots in~\(M\) differ by a sequence of crossing changes, then they are clearly homotopic. The converse is also widely known to hold, though a precise proof is difficult to come by in the literature. We will use it as a black box:

\begin{lem}[Homotopy via crossing changes]\label{homotopy_crossing_changes}
    If two knots \(K\), \(K'\subset M\) are homotopic, then \(K'\)~may be obtained from~\(K\) by a finite sequence of crossing changes and isotopies.
\end{lem}

\begin{prop}[Modifying boundary knots]\label{change_the_boundary_knots}
    Let \(S\subset M\) be an embedded oriented compact connected surface with \(n\ge 3\) boundary components \(K_1, \dots, K_n\). For any knots \(K_1', \dots, K_n'\subset M\) with each \(K_i'\)~homotopic to~\(K_i\), there is an embedded oriented surface \(S'\subset M\) diffeomorphic to~\(S\), whose boundary components are isotopic to \(K_1', \dots, K_n'\).
\end{prop}
\begin{proof}
    By \Cref{homotopy_crossing_changes}, it suffices to show that given any component~\(K\) of~\(\partial S\) and any knot~\(K'\subset M\) that differs from~\(K\) by one crossing change, it is possible to modify~\(S\) to a diffeomorphic surface~\(S'\) in such a way that \(K\)~is changed into a knot isotopic to~\(K'\), and all other knots in \(\partial S\) have their isotopy types unchanged. 
    By re-indexing the boundary components if necessary, we assume \(K=K_n\), and consider an embedded arc \(\alpha \subset M\) guiding a crossing change from~\(K_n\) to~\(K'\); in particular \(\alpha \cap K_n = \partial \alpha\). After a small isotopy of~\(\alpha\) relative endpoints, we may assume that \(\alpha\)~intersects~\(S\) transversely (in particular, \(\alpha \cap S\) consists of finitely many points), and that \(\alpha \cap \partial S =\partial \alpha\).
    
    Our first goal is to modify~\(S\) by an isotopy relative to~\(K_n\), to ensure \(\alpha \cap S = \partial \alpha\). So suppose \(p\in S\cap \alpha\) is an interior point of~\(\alpha\). We choose an embedded arc~\(\beta\subset S\) from~\(p\) to any point of~\(K_1\) and otherwise disjoint from~\(\partial S\) (note that \(K_1 \neq K_n\) since \(n\ge 2\)). Then we isotope \(S\) by dragging \(K_1\) along~\(\beta\), thus obtaining a surface with the same diffeomorphism type as~\(S\), and having fewer intersections with~\(\alpha\). This is illustrated in \Cref{fig:beta}. Applying this technique repeatedly, we remove all intersections of~\(\alpha\) with the interior of~\(S\).
    
    \begin{figure}[h]
        \centering
        \def\svgwidth{\linewidth}
        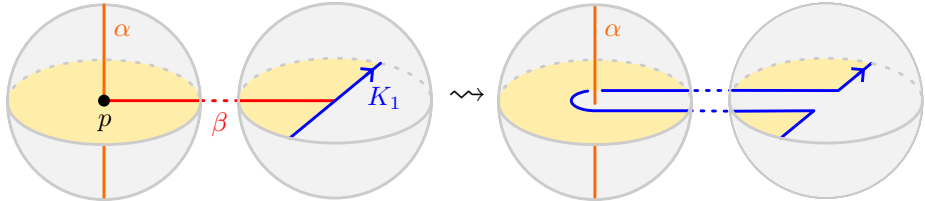
        \caption{Dragging \(K_1\) along~\(\beta\) to reduce the number of points in~\(\alpha\cap S\).}
        \label{fig:beta}
    \end{figure}

    Once the setup has been arranged so that \(\alpha\)~meets~\(S\) only at its endpoints, we choose a closed neighborhood~\(B\) of~\(\alpha\) diffeomorphic to a \(3\)-ball, within which the crossing change from~\(K_n\) to~\(K'\) will be executed. By taking \(B\)~small enough, we may assume that \(B \cap S\) is a pair of discs as illustrated in \Cref{fig:two_discs}. For each of these discs, we choose a point in its common boundary with~\(B\) (but disjoint from~\(K_n\)), and we call these points \(p_1\), \(p_2\).

    \begin{figure}[h!]
        \centering
        \def\svgwidth{0.4\linewidth}
        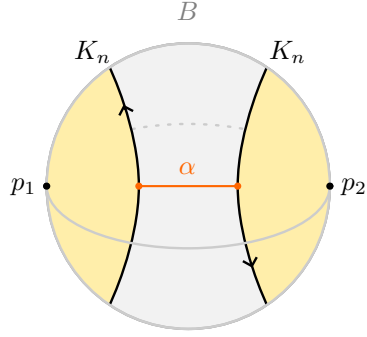
        \caption{The \(3\)-ball~\(B\) on which a crossing change from~\(K_n\) to \(K'\) will be executed.}
        \label{fig:two_discs}
    \end{figure}

    \begin{figure}[!b]
        \centering
        \def\svgwidth{\linewidth}
        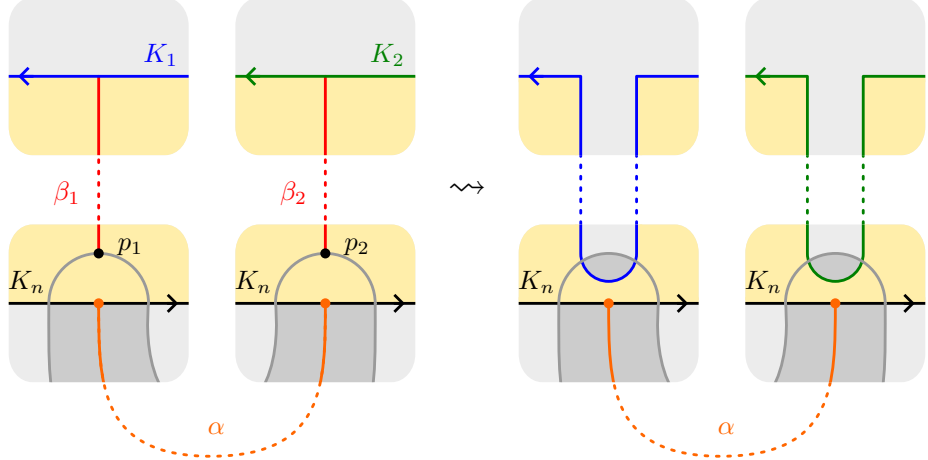
        \caption{Isotoping~\(S\) by pulling a strand of~\(K_1\) and one of~\(K_2\) into~\(B\) to turn the intersection \(S\cap B\) into a pair of strips.}
        \label{fig:create_strips}
    \end{figure}

    Consider the surface \(S_-\coloneq S\setminus \mathring B\).
    For each \(i\in \{1,2\}\), we choose a properly embedded arc \(\beta_i \subset S_-\) connecting~\(p_i\) to~\(K_i\).
    Since an arc connecting distinct boundary components of a surface is always non-separating, we may take \(\beta_1\)~and~\(\beta_2\) to be disjoint. Here we used that $n\geq 3$, to make sure that the boundary components of $S_-$ corresponding to $K_1, K_2, K_n$ are distinct. We isotope~\(S\) relative to~\(K_n\) by dragging each~\(K_i\) along~\(\beta_i\) until it intersects~\(B\) at a properly embedded arc, so \(S\cap B\) consists of a pair of bands. This isotopy is illustrated in \Cref{fig:create_strips}.

    Once the intersection of \(B\) with~\(S\) consists of two strips, we carry out the crossing change from~\(K_n\) to~\(K'\) bringing along the nearby strands of~\(K_1\) and~\(K_2\), see \Cref{fig:change_the_crossing}. This does not change the diffeomorphism type of~\(S\), nor the isotopy types of \(K_1\) nor \(K_2\) (though it may change the isotopy type of the link \(K_1 \cup K_2\)). All other boundary components of~\(S\) are unaffected. 
    \begin{figure}[h!]
        \centering
        \def\svgwidth{0.9\linewidth}
        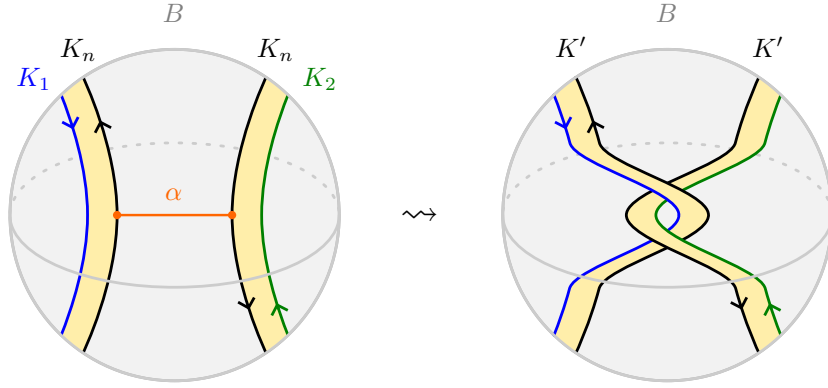
        \caption{Executing the crossing change from \(K_n\) to~\(K'\) while bringing along the two strips of~\(S\). The diffeomorphism type of~\(S\) is unchanged, as are the isotopy types of the knots \(K_1\) and \(K_2\).}
        \label{fig:change_the_crossing}
    \end{figure}
\end{proof}

\Cref{fig:trefoil_unknot_unknot} exemplifies how to start with a pair of pants in~\(\Sph^3\) with unknots as boundary, and modify it so that one of the boundary components becomes a trefoil.

\begin{figure}[h!]
        \centering
        \def\svgwidth{0.5\linewidth}
\begingroup%
  \makeatletter%
  \providecommand\color[2][]{%
    \errmessage{(Inkscape) Color is used for the text in Inkscape, but the package 'color.sty' is not loaded}%
    \renewcommand\color[2][]{}%
  }%
  \providecommand\transparent[1]{%
    \errmessage{(Inkscape) Transparency is used (non-zero) for the text in Inkscape, but the package 'transparent.sty' is not loaded}%
    \renewcommand\transparent[1]{}%
  }%
  \providecommand\rotatebox[2]{#2}%
  \newcommand*\fsize{\dimexpr\f@size pt\relax}%
  \newcommand*\lineheight[1]{\fontsize{\fsize}{#1\fsize}\selectfont}%
  \ifx\svgwidth\undefined%
    \setlength{\unitlength}{159.95105413bp}%
    \ifx\svgscale\undefined%
      \relax%
    \else%
      \setlength{\unitlength}{\unitlength * \real{\svgscale}}%
    \fi%
  \else%
    \setlength{\unitlength}{\svgwidth}%
  \fi%
  \global\let\svgwidth\undefined%
  \global\let\svgscale\undefined%
  \makeatother%
  \begin{picture}(1,0.8235929)%
    \lineheight{1}%
    \setlength\tabcolsep{0pt}%
    \put(0,0){\includegraphics[width=\unitlength,page=1]{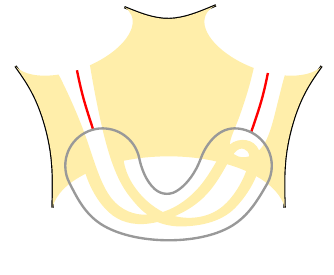}}%
    \put(0.8141561,0.63182032){\color[rgb]{1,0,0}\makebox(0,0)[t]{\lineheight{0.69999999}\smash{\begin{tabular}[t]{c}$\beta_2$\end{tabular}}}}%
    \put(0.47189017,0.11326683){\color[rgb]{1,0.4,0}\makebox(0,0)[t]{\lineheight{0.69999999}\smash{\begin{tabular}[t]{c}$\alpha$\end{tabular}}}}%
    \put(0.9443984,0.77222394){\color[rgb]{0,0.50196078,0}\makebox(0,0)[t]{\lineheight{0.69999999}\smash{\begin{tabular}[t]{c}$K_2$\end{tabular}}}}%
    \put(0,0){\includegraphics[width=\unitlength,page=2]{trefoil_unknot_unknot.pdf}}%
    \put(0.22774789,0.64353471){\color[rgb]{1,0,0}\makebox(0,0)[t]{\lineheight{0.69999999}\smash{\begin{tabular}[t]{c}$\beta_1$\end{tabular}}}}%
    \put(0.0558443,0.78160169){\color[rgb]{0,0,1}\makebox(0,0)[t]{\lineheight{0.69999999}\smash{\begin{tabular}[t]{c}$K_1$\end{tabular}}}}%
    \put(0.10273483,0.10873993){\color[rgb]{0,0,0}\makebox(0,0)[t]{\lineheight{0.69999999}\smash{\begin{tabular}[t]{c}$K_n$\end{tabular}}}}%
  \end{picture}%
\endgroup%

        \caption{Executing a crossing change on a boundary knot of a pair of pants in~\(\Sph^3\).}
        \label{fig:trefoil_unknot_unknot}
\end{figure}

\begin{thm}
[Genus-zero links of prescribed knots, \Cref{main_intro}]\label{main}
 Fix \(n\ge 3\) knots \(K_1, \dots, K_n\) in a connected \(3\)-manifold~\(M\), and for each~\(i\le n\) let \(C_i\subseteq\pi_1(M)\) be the conjugacy class represented by~\(K_i\). The following are equivalent:
    \begin{enumerate}
        \item The conjugacy classes \(C_1, \dots, C_n\)~form a canceling multiset.
        \item  There is an \(n\)-component genus-zero link whose components are isotopic to \(K_1 , \ldots , K_n\).
    \end{enumerate}  
\end{thm}

\begin{proof}
    The implication \((2 \Rightarrow 1)\) is a special case of \Cref{topological-cancellation}.

    For the converse, we begin by using \Cref{starting_surface} to construct an embedded genus-zero surface \(S_0 \subset M\) with \(n\)~boundary components representing the classes \(C_1, \dots, C_n\). In other words, the boundary components of~\(S_0\) are homotopic to the given~\(K_i\). We then apply \Cref{change_the_boundary_knots} to modify~\(S_0\) to a surface~\(S\) whose boundary components are isotopic to~\(K_1, \dots, K_n\). The boundary \(\partial S\) is thus the desired link.
\end{proof}

\begin{rem}[Band surgery]\label{rem:band_surgery}
    The case \(n=3\) of \Cref{main} may be restated in terms of (oriented) band surgery. It tells us that for every three knots \(K_1, K_2, K_3\) in~\(M\) whose conjugacy classes \(C_1, C_2, C_3 \subseteq \pi_1(M)\) form a canceling multiset, there are knots \(K_1'\),~\(K_2'\) isotopic to \(K_1\),~\(K_2\), and a band surgery between them that produces~\(K_3\rev\).

    For \(M=\Sph^3\), a special case is the statement that every knot is a band surgery of two unknots. This fact lies at the core of the magic trick explored by the authors in a different article \cite{AQ26}, and was indeed the observation that prompted the writing of this one. 
\end{rem}

\section{Linking and framing numbers}\label{sec:linking}

 When applying \Cref{main} to construct a genus-zero link $L$ of $n \geq 3$ prescribed knot types, the components of $L$ are in general intricately linked. In fact, if any component~\(K\) of~\(L\) can be separated from the others by a \(2\)-sphere~\(\Sigma\subset M\), then \(K\)~must be an unknot (this is seen by starting with a genus-zero Seifert surface~\(S\) of~\(L\) transverse to~\(\Sigma\), and capping off~\(S\) on both sides of its intersections with~\(\Sigma\), starting with the ones that are innermost in~\(\Sigma\)). 
However, restricting to the case where \(M=\Sph^3\), we will see that the components of~\(L\) may be chosen with pairwise vanishing linking numbers. The goal of this section is to prove a more refined version of this assertion.

Besides linking numbers, we will also consider the ``self-linking numbers'' induced from a Seifert surface~\(S\) for~\(L\) on each link component~\(K\).
By this we mean the \textbf{framing number}
\[ \operatorname{fr}_S(K) \coloneq \lk(K,K'),\]
where \(K'\)~is a push-off of~\(K\) into~\(S\), and \(\lk\)~denotes the linking number.
It turns out that once \(L\)~is fixed, these framing numbers are independent of~\(S\):

\begin{prop}[Linking numbers determine framings]\label{linking-framing-obstruction}
    Let \(S\) be an oriented compact connected surface smoothly embedded in~\(\Sph^3\), and let \(K_1, \dots, K_n\) be the components of~\(L\coloneq\partial S\). Then for every \(i\in\{1, \dots, n\}\) we have
    \[\fr_S(K_i) = -\!\!\sum_{\substack{1\le j\le n\\ j\neq i}} \lk(K_i, K_j).\]
\end{prop}
\begin{proof}
    Construct the exterior~\(X_L\) of~\(L\) as the complement of an  open tubular neighborhood of~\(L\) in~\(\Sph^3\) such that the intersection \(S \cap \partial X_L\) is transverse and consists of curves~\(K_i'\) parallel to the~\(K_i\). We denote by~\(S'\) the ``truncated'' Seifert surface \(S\cap X_L\), and by \(\mu_i \subseteq \partial X_L\) a meridian of each~\(K_i\).
    
    Recall that the homology classes~\([\mu_i]\) form a \(\ZZ\)-basis of~\(\h_1(X_L) \cong \ZZ^n\), and that the coordinates of a knot \(K\subset X_L\) in this basis are given by \([K] = \sum_{i=1}^n \lk(K_i,K)[\mu_i]\). Hence, in~\(\h_1(X_L)\),
    \[0 = [\partial S'] = \sum_{j=1}^n [K'_j] = \!\! \sum_{\substack{1\le i \le n\\1\le j \le n}}\lk(K_i,K'_j) [\mu_i].\]
    We thus conclude that for every $i\in\{1,\ldots,n\}$, 
    \[\sum_{j=1}^n \lk(K_i,K'_j)= 0.\]
    Replacing $K'_j$ with $K_j$ whenever $j\neq i$ yields the desired formula. 
\end{proof}

The following result tells us that every link may be adjusted to have any combination of linking numbers, without changing the knot types of its components, nor the Seifert genus. 

\begin{prop}[Modifying linking numbers]\label{adjust_linking}
    Let \(S_0\)~be an oriented compact connected surface embedded in~\(\Sph^3\) with boundary components \(K_1, \dots, K_n\), and for each distinct \(i,j\in \{1, \dots, n\}\) let \(\ell_{ij}\) be an integer, such that \(\ell_{ij} = \ell_{ji}\).
    Then there is an embedded surface \(S\subset \Sph^3\) diffeomorphic to~\(S_0\) with boundary components \(K_1', \dots, K_n'\), such that for all distinct \(i,j\in \{1, \dots,n\}\),
    \begin{enumerate}
        \item the knot~\(K_i'\) is isotopic to~\(K_i\), and
        \item \(\lk(K_i', K_j') = \ell_{ij}\).
    \end{enumerate}
\end{prop}
\begin{proof} 
    By addressing each pair \(i\neq j\) separately, we may assume without loss of generality that \(\ell_{ij}=\lk(K_i, K_j)\) for all but one pair of distinct indices~\(i,j\). 
    The surface \(S\) is constructed from~\(S_0\) by choosing a properly embedded arc \(\alpha\subset S\) connecting~\(K_i\) to~\(K_j\) and adding \(m \coloneq \ell_{ij}-\lk(K_i, K_j)\) full twists to a narrow strip around~\(\alpha\); this is illustrated in \Cref{fig:add_twist}. The resulting surface~\(S\) is diffeomorphic to~\(S_0\) (abstractly, \(S\)~differs from~\(S_0\) by the removal and reattachment of a band along the same pair of boundary arcs), 
    and the new knots \(K_i'\), \(K_j'\) satisfy \( \lk (K_i', K_j') = \lk(K_i,K_j) + m= \ell_{ij}\). The remaining linking numbers, as well as all other boundary components of~\(S_0\) are unchanged. Moreover, \(K_i'\)~is isotopic to~\(K_i\), and \(K_j'\) to~\(K_j\). \qedhere
    
    \begin{figure}[h]
    \centering
    \def\svgwidth{0.6\linewidth}
        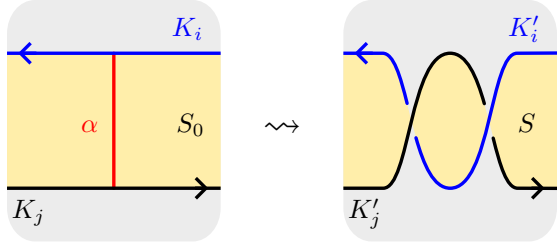
    \caption{Modifying~\(S_0\) to~\(S\) by adding \(m=1\) full twists to a strip around~\(\alpha\).}
    \label{fig:add_twist}
\end{figure}
\end{proof}

We obtain a version of \Cref{main_intro} for \(n\)-tuples of knots in~\(\Sph^3\), with additional control on linking numbers and framings.

\begin{cor}[Prescribing linking numbers and framings, \Cref{thm:linking_intro}]\label{full_linking_result}
    Let \(K_1, \dots, K_n\) be \(n \geq 3\) knots 
    in~\(\Sph^3\), and let \((\ell_{ij})_{i,j}\) be a symmetric \(n\times n\) matrix of integers. The following are equivalent:
    \begin{enumerate}
        \item For each \(i\in\{1, \dots, n\}\), we have
        \(\sum_{j=1}^n \ell_{ij}= 0\).

        \item There is a genus-zero link in \(\Sph^3\) with components \(K_1', \dots, K_n'\), such that for all \(i,j\in \{1, \dots, n\}\),
        \begin{enumerate}
            \item the knot \(K_i'\) is isotopic to~\(K_i\),
            \item if \(i \neq j \) then \(\lk(K_i', K_j') = \ell_{ij}\), and
            \item for some / every Seifert surface~\(S\) for~\(L\), we have \(\fr_S(K_i') = \ell_{ii}\). 
        \end{enumerate}
    \end{enumerate}
\end{cor}
\begin{proof}
    The implication \((2\Rightarrow 1)\) is a direct consequence of \Cref{linking-framing-obstruction}.

    For the converse, we begin by applying \Cref{main} to produce an oriented genus-zero surface \(S_0\subset \Sph^3\) with \(n\)~boundary components isotopic to the given knots~\(K_i\). Using \Cref{adjust_linking}, we modify~\(S_0\) to a diffeomorphic surface~\(S\) whose boundary components~\(K_1', \ldots, K_n'\) have the same isotopy types, and additionally satisfy \(\lk(K_i', K_j') = \ell_{ij}\) for all distinct \(i,j\). In other words, \(S\)~satisfies Conditions~(a) and~(b). Condition~1 and \Cref{linking-framing-obstruction} ensure that Condition~(c) is then automatically fulfilled.
\end{proof}

\section{The case \(n=2\)}\label{sec:n=2}

For this section, let us again fix a connected smooth \(3\)-manifold~\(M\). The \textbf{\(4\)-genus} of a link~\(L\) in~\(M\) is the smallest genus of an oriented compact connected properly embedded surface~\(S\subset M \times [0,1]\) such that \(\partial S = L\times \{1\}\). Denoting the (usual) genus and the \(4\)-genus of~\(L\) by~\(\g(L)\) and \(\g_4(L)\), respectively, we have \(\g_4(L)\le\g(L)\). A knot~\(K\) in~\(M\) is called \textbf{slice} if \(\g_4(K) = 0\). 

\Cref{main_intro} does not hold for \(n=2\), as two knots can only bound an annulus if one is the reverse of the other. We will see that it is nonetheless possible to obtain a similar result with only two knots as input, and a two-component link of vanishing \(4\)-genus as output, see \Cref{fig:ribbon_annulus} for an example.

\begin{figure}[h!]
    \centering
    \def\svgwidth{0.3 \linewidth}
    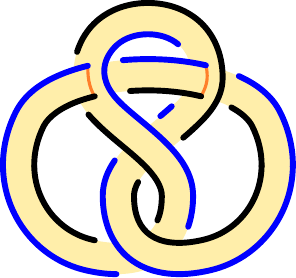
    \caption{A link of a trefoil (blue) and a figure-8 knot (black) in~\(\Sp3\) bounding an immersed annulus~\(A\). The  self-intersections are of ribbon-type, so they can be removed by perturbing \(A \subset \Sph^3\cong \Sp3\times\{1\}\) to be properly embedded in~\(\Sp3 \times [0,1]\).}
    \label{fig:ribbon_annulus}
\end{figure}

\begin{lem}[Removing a slice component]\label{remove-slice-component}
    Let \(L\subset M\) be a link with a component~\(K\) that is a slice knot. Then \[\g_4 (L \setminus K) \le \g(L).\]
\end{lem}
\begin{proof}
    Start with a minimal-genus Seifert surface \(S_0 \subset M\) for~\(L\). Since \(K\)~is slice, there is a properly embedded disk \(D \subset M\times [0,\tfrac 12]\) with \(\partial D = K\times \{\tfrac 12\}\).  After smoothening corners, the surface
    \[ S \coloneq \bigl( S_0 \times \{\tfrac 12\}\bigr) \cup D \cup \bigl( (L\setminus K) \times [\tfrac 12, 1]\bigr) \quad\subset M\times [0,1]\]
    has the same genus as~\(S_0\), and \(\partial S = (L\setminus K) \times \{1\}\). Thus \(S\)~witnesses the claimed inequality.
\end{proof}

\begin{cor}[Links with vanishing \(4\)-genus of prescribed knots, \Cref{maincor:zero_4genus}]\label{cor:zero_4genus}
    For all knots~\(K_1\), \(K_2\subset M\), the following are equivalent:
    \begin{enumerate}
        \item \(K_1\)~is homotopic to~\(K_2\rev\).
        \item There is a two-component link~\(L\subset M\) with components isotopic to~\(K_1\),~\(K_2\), such that \(\g_4(L)=0\). 
    \end{enumerate} 
\end{cor}
\begin{proof}
    The implication \((2 \Rightarrow 1)\) is a consequence of~\Cref{topological-cancellation}, together with the fact that the map \(M \cong M\times \{1\}\hookrightarrow M\times [0,1]\) is a homotopy equivalence.

    For the converse, we use \Cref{main} to produce a genus-zero \(3\)-component link~\(L_0\) in~\(M\), with components isotopic to \(K_1\), \(K_2\), and an unknot~\(U\) (or any slice knot). By \Cref{remove-slice-component}, the link \(L \coloneq L_0 \setminus U\) satisfies
    \(\g_4(L) \le \g(L_0) =0\).
\end{proof}

\begin{rem}[Prescribing linking numbers in~\(\Sp3\)]
For $M=\Sp3$, \Cref{cor:zero_4genus} states that for any two knots there exists a link~\(L\) with \(4\)-genus zero whose components are isotopic to the given knots.
Applying \Cref{adjust_linking} to the Seifert surface for the link~\(L_0\) in the proof, one can additionally ensure that the linking number of the components of~\(L\) is any prescribed integer.
\end{rem}

Another approach to the \(n=2\) case is to relax the genus-zero restriction and ask about the minimal genus of a \(2\)-component link with components isotopic to two given knots.
This leads back to work of Taniyama--Yasuhara \cite{TY94}, who showed that this quantity defines a distance on the set of isotopy classes of knots in~\(\Sph^3\), called the \textbf{C-distance}. They proved that  the C-distance is bounded from below by the cobordism distance (between concordance classes), and from above by the Gordian distance. 

In the same spirit, one may abandon the orientability assumptions and ask for the maximal Euler characteristic of a non-orientable surface with \(n=2\) boundary components of prescribed knots \cite[Section~5]{TY94}. Even the \(n=1\) case has received considerable attention, under the name ``crosscap number'' \cite{Cla78, KaLe16, ItTa20}.

\printbibliography

\flushleft
---------

\textsc{Raphael Appenzeller},  \texttt{\href{mailto:rappenzeller@mathi.uni-heidelberg.de}{rappenzeller@mathi.uni-heidelberg.de}}

\textsc{José Pedro Quintanilha},   \texttt{\href{mailto:jquintanilha@mathi.uni-heidelberg.de}{jquintanilha@mathi.uni-heidelberg.de}}

\vspace{8pt}
Universität Heidelberg,
Institut für Mathematik,\\
Im Neuenheimer Feld 205, 
69120 Heidelberg, Germany.

\end{document}